\documentclass[12pt,a4paper]{article}

\usepackage[dvips]{color}
\usepackage{color}
\usepackage{xcolor}

\usepackage{amsfonts}
\usepackage{amsfonts}
\usepackage{amsfonts}
\usepackage{amsfonts}
\usepackage{amsfonts}
\usepackage{mathrsfs}
\usepackage{amsfonts}
\usepackage{mathrsfs}
\usepackage{amsfonts}
\usepackage{amsfonts}
\usepackage{mathrsfs}
\usepackage{mathrsfs}
\usepackage{amsfonts}
\usepackage{amsfonts}
\usepackage{amsfonts}
\usepackage{amsfonts}
\usepackage{mathrsfs}

\usepackage{amsfonts}
\usepackage{mathrsfs}
\usepackage{mathrsfs}
\usepackage{mathrsfs}
\usepackage{mathrsfs}
\usepackage{amsfonts}
\usepackage{mathrsfs}
\usepackage{mathrsfs}
\usepackage{amsfonts}
\usepackage{amsfonts}
\usepackage{amsfonts}
\usepackage{amsfonts}
\usepackage{amsfonts}
\usepackage{amsfonts}
\usepackage{amsfonts}
\usepackage{amsfonts}
\usepackage{amsfonts}
\usepackage{mathrsfs}
\usepackage{amsfonts}
\usepackage{mathrsfs}
\usepackage{amsfonts}
\usepackage{mathrsfs}
\usepackage{amsfonts}
\usepackage{amsfonts}
\usepackage{amsfonts}
\usepackage{amsfonts}
\usepackage{amsfonts}
\usepackage{amsfonts}
\usepackage{amsfonts}
\usepackage{amsfonts}
\usepackage{amsfonts}
\usepackage{amsfonts}
\usepackage{amsfonts}
\usepackage{amsfonts}
\usepackage{mathrsfs}
\usepackage{mathrsfs}
\usepackage{mathrsfs}
\usepackage{mathrsfs}
\usepackage{mathrsfs}
\usepackage{mathrsfs}
\usepackage{mathrsfs}
\usepackage{mathrsfs}
\usepackage{mathrsfs}
\usepackage{amsfonts}
\usepackage{amsfonts}
\usepackage{amsfonts}
\usepackage{amsfonts}
\usepackage{amsfonts}
\usepackage{amsfonts}
\usepackage{mathrsfs}
\usepackage{amsfonts}
\usepackage{amsfonts}
\usepackage{amsfonts}
\usepackage{amsfonts}
\usepackage{amsfonts}
\usepackage{amsfonts}
\usepackage{amsfonts}
\usepackage{amsfonts}
\usepackage{amsfonts}
\usepackage{amsfonts}
\usepackage{amsfonts}
\usepackage{amsfonts}
\usepackage{mathrsfs}
\usepackage{amsfonts}
\usepackage{mathrsfs}
\usepackage{amsfonts}
\usepackage{amsfonts}
\usepackage{amsfonts}
\usepackage{mathrsfs}
\usepackage{amsfonts}
\usepackage{amsfonts}
\usepackage{amsfonts}
\usepackage{amsfonts}
\usepackage{amsmath}
\usepackage{mathrsfs}
\usepackage{amsfonts}
\usepackage{amsfonts}
\usepackage{amsfonts}
\usepackage{amsfonts}
\usepackage{amsfonts}
\usepackage{amsfonts}
\usepackage{amsfonts}
\usepackage{amsfonts}
\usepackage{amsfonts}
\usepackage{amsfonts}
\usepackage{amsfonts}
\usepackage{mathrsfs}
\usepackage{amsfonts}
\usepackage{amsfonts}
\usepackage{amsfonts}
\usepackage{amsfonts}
\usepackage{amsmath}
\usepackage{mathrsfs}
\usepackage{mathrsfs}
\usepackage{mathrsfs}
\usepackage{mathrsfs}
\usepackage{mathrsfs}
\usepackage{mathrsfs}
\usepackage{mathrsfs}
\usepackage{mathrsfs}
\usepackage{amsfonts}
\usepackage{amsfonts}
\usepackage{amsfonts}
\usepackage{amsfonts}
\usepackage{amsfonts}
\usepackage{amsfonts}
\usepackage{amsfonts}
\usepackage{amsfonts}
\usepackage{amsfonts}
\usepackage{amsfonts}
\usepackage{amsfonts}
\usepackage{amsfonts}
\usepackage{mathrsfs}
\usepackage{amsfonts}
\usepackage{mathrsfs}
\usepackage{amsfonts}
\usepackage{amsfonts}
\usepackage{amsfonts}
\usepackage{amsfonts}
\usepackage{amsfonts}
\usepackage{amsfonts}
\usepackage{amsfonts}
\usepackage{amsfonts}
\usepackage{amsfonts}
\usepackage{amsfonts}
\usepackage{amsfonts}
\usepackage{amsfonts}
\usepackage{amsfonts}
\usepackage{amsfonts}
\usepackage{amsfonts}
\usepackage{amsfonts}
\usepackage{amsfonts}
\usepackage{amsfonts}
\usepackage{mathrsfs}
\usepackage{mathrsfs}
\usepackage{mathrsfs}
\usepackage{mathrsfs}
\usepackage{amsfonts}
\usepackage{amsfonts}
\usepackage{amsfonts}
\usepackage{amsfonts}
\usepackage{amsfonts}
\usepackage{amsfonts}
\usepackage{amsfonts}
\usepackage{amsfonts}
\usepackage{amsfonts}
\usepackage{amsfonts}
\usepackage{amsfonts}
\usepackage{amsfonts}
\usepackage{amsfonts}
\usepackage{amsfonts}
\usepackage{amsfonts}
\usepackage{amsfonts}
\usepackage{amsfonts}
\usepackage{mathrsfs}
\usepackage{mathrsfs}
\usepackage{amsfonts}
\usepackage{amsfonts}
\usepackage{amsfonts}
\usepackage{amsfonts}
\usepackage{amsfonts}
\usepackage{amsfonts}
\usepackage{amsfonts}
\usepackage{amsfonts}
\usepackage{amsfonts}
\usepackage{amsfonts}
\usepackage{amsfonts}
\usepackage{amsfonts}
\usepackage{amsfonts}
\usepackage{mathrsfs}
\usepackage{mathrsfs}
\usepackage{amsfonts}
\usepackage{amsfonts}
\usepackage{amsfonts}
\usepackage{amsfonts}
\usepackage{amsfonts}
\usepackage{amsfonts}
\usepackage{mathrsfs}
\usepackage{mathrsfs}
\usepackage{amsfonts}
\usepackage{amsfonts}
\usepackage{amsfonts}
\usepackage{amsfonts}
\usepackage{amsfonts}
\usepackage{amsfonts}
\usepackage{amsfonts}
\usepackage{amsfonts}
\usepackage{amsfonts}
\usepackage{amsfonts}
\usepackage{amsfonts}
\usepackage{amsfonts}
\usepackage{amsfonts}
\usepackage{amsfonts}
\usepackage{mathrsfs}
\usepackage{amsfonts}
\usepackage{amsfonts}
\usepackage{amsfonts}
\usepackage{amsfonts}
\usepackage{amsfonts}
\usepackage{amsfonts}
\usepackage{amsfonts}
\usepackage{amsfonts}
\usepackage{amsfonts}
\usepackage{amsfonts}
\usepackage{amsfonts}
\usepackage{amsfonts}
\usepackage{amsfonts}
\usepackage{amsfonts}
\usepackage{amsfonts}
\usepackage{amsfonts}
\usepackage{amsmath}
\usepackage{mathrsfs}
\usepackage{amsfonts}
\usepackage{amsfonts}
\usepackage{amsfonts}
\usepackage{amsfonts}
\usepackage{amsmath}
\usepackage{mathrsfs}
\usepackage{amsfonts}
\usepackage{mathrsfs}
\usepackage{mathrsfs}
\usepackage{mathrsfs}
\usepackage{mathrsfs}
\usepackage{amsmath}
\usepackage{mathrsfs}
\usepackage{amsfonts}
\usepackage{color}
\usepackage{mathrsfs}

\usepackage{stmaryrd}
\usepackage{amsmath}
\usepackage{amssymb}

\usepackage{bbm}
\usepackage{mathrsfs}
\usepackage{graphicx}

\usepackage{enumerate}
\usepackage{thm}

\makeatletter
\def\tank#1{\protected@xdef\@thanks{\@thanks
        \protect\footnotetext[0]{#1}}}
\def\bigfoot{

    \@footnotetext}
\makeatother

\newcommand{\ea}{\end{array}}
\newtheorem{thm}{Theorem}[section]
\newtheorem{prop}{Proposition}[section]

\newtheorem{rmk}{Remark}[section]

\newtheorem{ass}{Assumption}[section]

\numberwithin{equation}{section}

{\theorembodyfont{\rmfamily}

\newenvironment{proof}{Proof}{\hfill $\Box$}

\def\RR{\mathbb{R}}
\def\PP{\mathbb{P}}

\def\et{{\eta}}

\def\al{{\alpha}}

\def\ga{{\gamma}}

\def\al{{\alpha}}

\setcounter{equation}{0}

\allowdisplaybreaks

\begin{document}

\title{\Large \bf Large deviation principles for SDEs under locally weak monotonicity conditions}
\date{}
\author{{Jian Wang}$^1$\footnote{E-mail:wg1995@mail.ustc.edu.cn}~~~{Hao Yang}$^1$\footnote{E-mail:yhaomath@ustc.edu.cn}~~~ {Jianliang Zhai}$^1$\footnote{E-mail:zhaijl@ustc.edu.cn}~~~ {Tusheng Zhang}$^{2}$\footnote{E-mail:Tusheng.Zhang@manchester.ac.uk}
\\
 \small  1. School of Mathematical Sciences,
 \small  University of Science and Technology of China,\\
 \small  Hefei, Anhui 230026, China.\\
 \small  2. Department of Mathematics, University of Manchester,\\
 \small  Oxford Road, Manchester, M13 9PL, UK.
}

\maketitle

\begin{center}
\begin{minipage}{130mm}
{\bf Abstract:}
This paper establishes a Freidlin-Wentzell large deviation principle for stochastic differential equations(SDEs) under locally weak monotonicity conditions and Lyapunov conditions. We illustrate the main result of the paper by showing that it can be applied to SDEs with non-Lipschitzian coefficients, which can not be covered in the existing literature. These  include the interesting biological  models like  stochastic Duffing-van der Pol oscillator model, stochastic SIR model, etc.

\vspace{3mm} {\bf Keywords:}
Freidlin-Wentzell large deviation principle; locally weak monotonicity condition; non-Lipschitzian coefficients.

\vspace{3mm} {\bf AMS Subject Classification:}
60F10; 60H10; 60H35.
\end{minipage}
\end{center}

\newpage

\renewcommand\baselinestretch{1.2}
\setlength{\baselineskip}{0.28in}
\section{Introduction and motivation}\label{Intr}
The small noise large deviation principle (LDP) for stochastic differential equations (SDEs) has a long history and has been studied by many authors. The general LDP was first
formulated by Varadhan \cite{vara} in 1966. Then after the pioneering work of the LDP on Markov process \cite{don&vara} and dynamical systems \cite{F-W} in the 1970s and 1980s, respectively, the LDP has attracted considerable attention as it deeply reveals the rules of the extreme events in risk management, statistical mechanics, informatics, quantum physics and so on.

SDEs with non-Lipschitzian coefficient often resulted in describing  important models in physics, engineering, biology etc. They are of constant interest. In this paper, we are concerned with the LDP for SDEs with non-Lipschitzian coefficient satisfying  locally weak monotonicity conditions  and Lyapunov conditions, see the precise assumptions  in Section 2. Our results can be applied to  a number of interesting biological models with nonlinear coefficients, such as the stochastic oscillation model, stochastic SIR model, stochastic Lorenz equation etc (see also [8]). 

As  we are aware of, there are not many results on the Freidlin-Wentzell type LDP for SDEs with non-Lipschitzian  coefficients.
Let us mention some. In 2005, the fourth named author and  Fang \cite{fang-zhang} used discretization and exponential equivalence  to establish the LDP for a class of SDE with  log-Lipschitz condition. Based on the weak convergence method,  the authors in \cite{R-Z&JFA} and \cite{R-Z&BSM} also obtained the LDP for SDEs under the log-Lipschitz condition. In 2020, Cheng and Huang \cite{C-H} studied the LDP for SDEs with the Dini continuity of the drift.  However,  the diffusion coefficients are required to be non-degenerate condition. We also refer the readers to \cite{rockner,lan,Z}  for related results.

 The existing results could not cover the situations where both the drifts and the diffusion coefficients are only H\"{o}lder continuous. In this paper, we establish a LDP for a class of SDE with locally weak
monotonicity conditions and Lyapunov conditions, see \eqref{21} and  \eqref{22} for the precise assumptions. Our framework is sufficiently general to include the cases of  H\"{o}lder continuous coefficients and  also can be applied to many interesting biological models, see exaples in Section 5.

To obtain our result, we will adopt the weak convergence approach introduced by Budhiraja, Dupuis and Maroulas in
\cite{dupuis} and \cite{BDM}. This is a challenging task because the locally weak monotonicity  coefficients can be  quite irregular and the Gronwall type equalities are not applicable. One of the difficulties is to prove that  the solutions of controlled equations converge to the solution of the skeleton in probability. This is mainly caused by the locally weak monotonicity condition, see  \eqref{21}. To overcome the difficulty, we construct some special control functions, apply stopping time techniques and use  a particularly suitable  sufficient condition proved in  \cite{zhang} to verify the criteria of Budhiraja-Dupuis-Maroulas.

The organization of the paper is as follows. In Section 2, we introduce the precise framework and the main result. Section 3
and Section 4 are devoted to the proof of the LDP for Eq. \eqref{12}. Some illustrating examples  are given in Section 5.

\section{The framework and main results}

Throughout this paper,  we will use the following notation. Let $(\mathbb{R}^d, \langle \cdot, \cdot \rangle, |\cdot|)$ be the \emph{d}-dimensional Euclidean space with the inner product $\langle \cdot, \cdot \rangle$ which induces the norm $|\cdot|$. The norm $\| \cdot \|$ stands for the Hilbert-Schmidt norm $\| \sigma \|^2 :=  \sum_{i=1}^{d} {\sum_{j=1}^{m}{\sigma_{ij}^2}}$ for any $d \times m$-matrix $\sigma = (\sigma_{ij}) \in \mathbb{R}^d \otimes \mathbb{R}^m$. ${A}^T$ stands for the transpose of the martix $A$. $A\cdot x$ denotes the product of the  matrix
$A$ and the vector $x$.

Let $(\Omega,\mathcal{F},\{{\mathcal{F}}_t\}_{t\geq 0},\mathbb{P})$  be a complete probability space with a filtration $\{{\mathcal{F}}_t\}_{t\geq 0}$ satisfying the usual conditions and $\{B(t)\}_{t \geq 0}$  a \emph{m}-dimensional Brownian motion on this probability space. Fix $T\in(0,\infty)$. Consider the following stochastic differential equation:
\begin{equation}\label{12}
dx^\epsilon(t)=b(t,x^\epsilon(t))dt+\sqrt{\epsilon}\sigma(t,x^\epsilon(t))dB(t),\ t\in[0,T], \quad x(0)=x_0,
\end{equation}
where the initial data $x_0 \in \mathbb{R}^d$. $\sigma : (t, x) \in [0, T] \times \mathbb{R}^d \mapsto \sigma (t,x) \in \mathbb{R}^d\otimes \mathbb{R}^m$ and $ b: (t,x) \in [0, T] \times \mathbb{R}^d \mapsto b(t,x) \in \mathbb{R}^d$ are Borel measurable functions which are continuous with respect to the space variable $x$.

Let us now introduce the following assumptions. Let $f,~g$ be nonnegative integrable functions on $[0,T]$.
\begin{ass}\label{ass0}
For arbitrary $R>0$,
\begin{equation*}
\int_0^T\sup_{|x|\leq R}(|b(s,x)|+\|\sigma(s,x)\|^2)ds<\infty.
\end{equation*}
\end{ass}
 \begin{ass}\label{ass1}
There exists $\epsilon_0\in(0,1)$ such that for arbitrary $R>0$, if $|x|\vee|y|\leq R$, $|x-y|\leq \epsilon_0$, the following locally weak monotonicity condition holds
\begin{equation}\label{21}
2\langle x-y,b(s,x)-b(s,y)\rangle+\|\sigma(s,x)-\sigma(s,y)\|^2\leq g(s)\eta_R(|x-y|^2),\ \forall s\in[0,T],
\end{equation}
where $\eta_R:[0,1)\rightarrow \mathbb R_+$ is an increasing, continuous function satisfying
\begin{equation}\label{24}
\eta _R(0) = 0,\quad \int_{0^+}\frac{dx}{\eta_R(x)}=+\infty.
\end{equation}
\end{ass}

\begin{ass}\label{ass2}
There exists $V\in C^2(\mathbb R^d;\mathbb R_+)$  and $\delta>0,~\eta>0$ such that
\begin{equation}\label{As b V}
\lim_{|x|\rightarrow+\infty}V(x)=+\infty,
\end{equation}
\begin{equation}\label{22}
\langle b(s,x),V_x(x)\rangle+\frac{\delta}{2}trace\big(V_{xx}(x)\sigma(s,x)\sigma^T(s,x)\big)+\frac{|\langle\sigma(s,x),V_x (x)\rangle|^2}{\eta V(x)}\leq f(s)\big(1+\gamma(V(x))\big),
\end{equation}
and
\begin{equation}\label{23}
trace\big(V_{xx}(x)\sigma(s,x)\sigma^T(s,x)\big)\geq 0,
\end{equation}
here $\gamma:[0,+\infty)\rightarrow\mathbb R_+$  is a continuous, increasing function satisfying
\begin{equation}\label{26}
\int_0^{+\infty}\frac{1}{\gamma(s)+1}ds=+\infty.
\end{equation}
Here $V_x$ and $V_{xx}$ stand for the first derivative and second derivative respectively.
\end{ass}

\begin{ass}\label{ass3} For any $0\leq c \leq 1$,
\begin{equation*}
\sup_{s\in [0, \epsilon_0]}\frac{c\eta_R(s)}{\eta_R(cs)}<\infty,\quad\ \sup_{s\in [0, \infty)}\frac{c\gamma(s)}{\gamma(cs)}<\infty.
\end{equation*}
Here $\epsilon_0$ is the constant appearing in Assumption \ref{ass1}.
\end{ass}

\begin{rmk}\label{re1}
The examples of the function {\rm  $\eta_R(s)$ in Assumption \ref{ass1} include $R s\log\frac{1}{s}$ and the examples of the function  $\gamma(s)$ in Assumption \ref{ass2}  include $s\log s+1$, etc..}
\end{rmk}
The following proposition gives the solution of the perturbation Eq. \eqref{12}. The proof is similar to that of Theorem 1.1 in \cite{wu}.
\begin{prop}\label{prop1}
For any $0<\epsilon<1$, under Assumptions \ref{ass0}, \ref{ass1} and \ref{ass2}, there exists a unique strong solution of Eq. \eqref{12}.
\end{prop}
For each $h\in L^2([0,T],\mathbb R^m)$, consider the so called skeleton equation:
\begin{equation}\label{30}
dx^h(t)=b(t,x^h(t))dt+\sigma(t,x^h(t))\cdot h(t)dt,
\end{equation}
with the initial data $x^h(0)=x_0$.
We have the following result:
\begin{prop}\label{prop2}
Under Assumptions \ref{ass0}, \ref{ass1} and \ref{ass2}, there exists a unique solution to Eq. \eqref{30}.
\end{prop}
The proof of this proposition is similar to the proof of Theorem 1.1 in \cite{wu}, so we omit it here.

We now formulate the main result in this paper.
\begin{thm}\label{thm}
For $\epsilon>0$, let $X^\epsilon$ be the solution to Eq. \eqref{12}. Suppose Assumptions \ref{ass0}, \ref{ass1}, \ref{ass2} and \ref{ass3}  are satisfied, then the family $\{X^\epsilon\}_{\epsilon>0}$ satisfies a large deviation principle on the space $C([0,T],\mathbb{R}^d)$ with the rate function $I:C([0,T],\mathbb{R}^d)\rightarrow [0,\infty]$, where
\begin{equation}\label{rate}
I(y)=\inf_{\left\{h \in L^2([0,T]; \mathbb{R}^m):\  y=x^h \right\}}\big\{\frac{1}{2}
\int_0^T|h(s)|^2ds\big\},
\end{equation}
with the convention $\inf\{\emptyset\}=\infty$, here $x^h\in C([0,T],\mathbb{R}^d)$ solves Eq. (\ref{30}).
\end{thm}


\begin{proof}
According to Proposition \ref{prop2}, there exists  a measurable mapping $\varGamma^0(\cdot):C([0,T], \mathbb{R}^m)\rightarrow C([0,T], \mathbb{R}^d)$ such that $x^{h} = \varGamma^0 \big(\int_0^{\cdot}{h(s)ds}\big)$ for  $h\in L^2([0,T],\mathbb R^m)$.

Set
\begin{equation*}
S^N :=\{h\in L^2([0,T],\mathbb R^m):|h|_{L^2([0,T],\mathbb R^m)}^2\leq N\},
\end{equation*}
and
\begin{equation*}
\tilde{S}^N :=\{\phi : \phi \ is \ \mathbb{R}^m\text{-}valued \  {\mathcal{F}}_t\text{-}predictable\  process\  such\ that \  \phi(\omega) \in S^N,\  \mathbb{P}\text{-}a.s.\}.
\end{equation*}
Throughout this paper,  $S^N$ is endowed with the weak topology on $L^2([0,T],\mathbb R^m)$, under which $S^N$ is a compact Polish space.

By the Yamada-Watanabe theorem, the existence of a unique strong solution of Eq. \eqref{12} implies that for
every $\epsilon>0$, there exists a measurable mapping $\varGamma^{\epsilon}(\cdot):C([0,T], \mathbb{R}^m)\rightarrow C([0,T], \mathbb{R}^d)$ such that
\begin{equation*}
X^{\epsilon}=\varGamma^{\epsilon}( B(\cdot) ),
\end{equation*}
and applying the Girsanov theorem, for any $N>0$ and $h^{\epsilon}\in \tilde{S}^N$,
\begin{equation}\label{52}
Y^{\epsilon} := \varGamma^{\epsilon}\big( B(\cdot) + \frac{1}{\sqrt{\epsilon}}\int_0^{\cdot}{h^{\epsilon}(s)ds}\big)
\end{equation}
is the solution of the following SDE
\begin{equation}\label{32}
Y^{\epsilon}(t)=x+\int_0^tb(s,Y^{\epsilon}(s))ds+\int_0^t\sigma(s,Y^{\epsilon}(s))\cdot h^\epsilon(s)ds+\sqrt{\epsilon}\int_0^t\sigma(s,Y^{\epsilon}(s))dB(s).
\end{equation}
According to Theorem 3.2 in \cite{zhang}, Theorem \ref{thm} is established once we have proved:

(i) for every $N < +\infty$ and any family $\{h_n, n\in\mathbb{N}\} \subset S^N$ converging  weakly to some element $h$ as $n \rightarrow \infty$, $\varGamma^0 \big(\int_0^{\cdot}{h_n}(s)ds\big)$ converges to $\varGamma^0 \big(\int_0^{\cdot}{h(s)ds}\big)$ in the space $C([0,T],\mathbb{R}^d)$.

(ii) for every $N < +\infty$ and any family $\{h^{\epsilon}, \epsilon > 0\} \subset \tilde{S}^N$ and any $\delta > 0$,
\[ \lim_{\epsilon \rightarrow 0} \mathbb{P}\big(\rho(Y^{\epsilon},Z^{\epsilon})> \delta \big) = 0,\]
where $Z^{\epsilon}=\varGamma^0 \big(\int_0^{\cdot}{h^{\epsilon}(s)ds}\big)$ and $\rho(\cdot,\cdot)$ stands for the uniform metric in the space $C([0,T],\mathbb{R}^d)$.

Statement  (i) is verified in Proposition \ref{prop3} in Section 3 and statement (ii) is established in Proposition \ref{prop4} in Section 4.
\end{proof}

In the sequel, the symbol C will denote a positive generic constant whose value may change from place to place.

\section{Proof of statement (i)}
In this section, we will prove the following result.
\begin{prop}\label{prop3}
Suppose  Assumptions \ref{ass0}, \ref{ass1}, \ref{ass2} and \ref{ass3}  are satisfied and assume that $h_n,n\in\mathbb{N},~h\in S^N$, $\lim_{n\rightarrow\infty}h_n= h$ in the weak topology of $L^2([0,T],\mathbb R^m)$. Then $\lim_{n\rightarrow\infty}x^{h_n}= x^h$ in the space $C([0,T],\mathbb R^d)$, where $x^h$ solve Eq. \eqref{30}, and $x^{h_n}$ solve Eq. \eqref{30} with $h$ replaced by $h_n$.
\end{prop}
\begin{proof}
The proof is divided into two steps.
 \vskip 0.1cm

{\bf Step 1:} We prove that the family $\{x^{h_n},n\in\mathbb{N}\}$ is tight in $C([0,T],\mathbb R^d)$.
 \vskip 0.1cm

Define $\varphi(y)=\int_0^y\frac{1}{\gamma(s)+1}ds$ and $w^{h_n}(t)=e^{-\eta\int_0^{t}|h_n(s)|^2ds}V(x^{h_n}(t))$, where $\eta > 0,~\gamma,$ and $V$ are in Assumption \ref{ass2}. Apply the chain rule to get
\begin{eqnarray}
 &&\varphi(w^{h_n}(t))\nonumber\\
 &=& \varphi(V(x_0))+\int_0^{t} \varphi^\prime(w^{h_n}(s))\cdot e^{-\eta\int_0^s|h_n(r)|^2dr} \cdot\big[-\eta|h_n(s)|^2V(x^{h_n}(s)) \nonumber\\
  &&+\langle V_x(x^{h_n}(s)),b(s,x^{h_n}(s))\rangle
+\langle V_x(x^{h_n}(s)),\sigma(s,x^{h_n}(s))\cdot h_n(s)\rangle\big]ds\nonumber\\
  &\leq& \varphi(V(x_0))+\int_0^{t} \varphi^\prime(w^{h_n}(s))\cdot e^{-\eta\int_0^s|h_n(r)|^2dr} \cdot\big[\langle V_x(x^{h_n}(s)),b(s,x^{h_n}(s))\rangle \nonumber\\
  &&+\frac{|\langle\sigma(s,x^{h_n}(s)),V_x (x^{h_n}(s))\rangle|^2}{\eta V(x^{h_n}(s))}\big]ds\nonumber\\
  &\leq& \varphi(V(x_0))+\int_0^{t} \varphi^\prime(w^{h_n}(s))\cdot e^{-\eta\int_0^s|h_n(r)|^2dr} \cdot f(s)\big(1+\gamma(V(x^{h_n}(s)))\big)ds\nonumber\\
   &\leq& \varphi(V(x_0))+C\int_0^Tf(s)ds.\label{34}
\end{eqnarray}
Assumption \ref{ass3} has been used in getting the last inequality.

 The above inequality (\ref{34}) yields
\begin{equation*}
  \sup_{n\in\mathbb{N}}\sup_{t\in[0,T]}V(x^{h_{n}}(t))<\infty.
\end{equation*}
By the condition (\ref{As b V}) on the function $V$, we deduce that
\begin{equation}\label{37}
  \sup_{n\in\mathbb{N}}\sup_{t\in[0,T]}|x^{h_n}(t)|\leq L
\end{equation}
for some constant $L>0$.

From the proof of (\ref{37}) we also see  that
\begin{eqnarray}\label{eq bound}
\sup_{\chi\in S^N}\sup_{t\in[0,T]}|x^{\chi}(t)|\leq L_1
\end{eqnarray}
for some constant $L_1$. This fact will be used in next section.

For $0\leq s<t\leq T$,
\begin{equation*}
x^{h_n}(t)-x^{h_n}(s)=\int_s^tb(r,x^{h_n}(r))dr+\int_s^t\sigma(r,x^{h_n}(r))\cdot h_n(r)dr.
\end{equation*}
By the H\"{o}lder inequality and \eqref{37} we have
\begin{eqnarray*}
   |x^{h_n}(t)-x^{h_n}(s)|
&\leq&
  \int_s^t\sup_{|x|\leq L}|b(r,x)|dr
  +
  \big(\int_s^t\sup_{|x|\leq L}\|\sigma(r,x)\|^2dr\big)^{\frac{1}{2}}\big(\int_s^t|h_n(r)|^2dr\big)^{\frac{1}{2}}\\
&\leq&
   \int_s^t\sup_{|x|\leq L}|b(r,x)|dr
  +
  \big(\int_s^t\sup_{|x|\leq L}\|\sigma(r,x)\|^2dr\big)^{\frac{1}{2}}N^{\frac{1}{2}}.
\end{eqnarray*}
By  Assumption \ref{ass0} it follows that, for any $\kappa>0$, there exists $\theta>0$ such that for any $0\leq s<t\leq T$ with $t-s\leq \theta$,
\begin{equation}\label{38}
  \sup_{n\in\mathbb{N}}|x^{h_n}(t)-x^{h_n}(s)|\leq \kappa.
\end{equation}
By \eqref{37}, \eqref{38}, and the Arzela-Ascoli theorem, $\{x^{h_n},n\in\mathbb{N}\}$ is tight in the space $C([0,T],\mathbb R^d)$. Hence there exists a subsequence of $x^{h_n}$ (still denoted as $x^{h_n}$)  and $\tilde{x}\in C([0,T],\mathbb{R}^d)$ satisfying
\begin{equation}\label{39}
  x^{h_n}\rightarrow \tilde{x} \quad \text{in}~C([0,T],\mathbb{R}^d).
\end{equation}

\vskip 0.1cm
{\bf Step 2:} We verify that $\tilde{x}=x^h$, completing the proof.

Recall that $x^{h_n}$ is the unique solution of
\begin{equation}\label{eq u hn}
x^{h_n}(t)=x+\int_0^tb(s,x^{h_n}(s))ds+\int_0^t\sigma(s,x^{h_n}(s))\cdot h_n(s)ds,\quad t\in[0,T].
\end{equation}
By the dominated convergence theorem and the continuity of $b$ with respect to the second variable,
\begin{equation}\label{41}
  \lim_{n\rightarrow\infty}\int_0^t|b(s,x^{h_n}(s))-b(s,\tilde{x}(s))|ds=0.
\end{equation}

Next, we prove
\begin{eqnarray}\label{eq limit sigma}
\lim_{n\rightarrow\infty}\int_0^t\big[\sigma(s,x^{h_n}(s))\cdot h_n(s)-\sigma(s,\tilde{x}(s))\cdot h(s)\big]ds=0.
\end{eqnarray}

Note that
\begin{eqnarray}\label{eq sigma}
  && \big|\int_0^t\big[\sigma(s,x^{h_n}(s))\cdot h_n(s)-\sigma(s,\tilde{x}(s))\cdot h(s)\big]ds\big|\nonumber\\
  &\leq&\big|\int_0^t\big(\sigma(s,x^{h_n}(s))-\sigma(s,\tilde{x}(s))\big)\cdot h_n(s)ds\big|+
  \big|\int_0^t\sigma(s,\tilde{x}(s))\cdot (h_n(s)-h(s))ds\big|\nonumber\\
 &=&I_1+I_2.
\end{eqnarray}
 By the H\"{o}lder inequality, the dominated convergence theorem and \eqref{39},
 \begin{eqnarray}\label{eq sigma 1}
 I_1
 &\leq&
 \big(\int_0^t\|\sigma(s,x^{h_n}(s))-\sigma(s,\tilde{x}(s))\|^2ds\big)^{\frac{1}{2}}\big(\int_0^t|h_n(s)|^2ds\big)^{\frac{1}{2}}\nonumber\\
  &\leq&
  \big(\int_0^t\|\sigma(s,x^{h_n}(s))-\sigma(s,\tilde{x}(s))\|^2ds\big)^{\frac{1}{2}}N^{\frac{1}{2}}
 \rightarrow 0,\ \ \text{as }n\rightarrow\infty.
 \end{eqnarray}
Let $e_i$, $1\leq i \leq d$, be the canonical basis of $\RR^d$.
 Since $h_n\rightarrow h$ weakly in $L^2([0,T],\mathbb{R}^d)$, we derive that for each $1\leq i \leq d$,
\begin{equation*}
\langle\int_0^t\sigma(s,\tilde{x}(s))\cdot (h_n(s)-h(s))ds,e_i\rangle\rightarrow 0,
 \end{equation*}
 which further implies that $\lim_{n\rightarrow\infty}I_2=0$. Combining (\ref{eq sigma}) and (\ref{eq sigma 1}), we obtain (\ref{eq limit sigma}).

Now Let $n\rightarrow\infty$ in (\ref{eq u hn}) to see that $\tilde{x}$ is the solution of the equation:
\begin{equation}\label{40}
\tilde{x}(t)=x+\int_0^tb(s,\tilde{x}(s))ds+\int_0^t\sigma(s,\tilde{x}(s))\cdot h(s)ds,\quad t\in[0,T].
\end{equation}
By the uniqueness,  we have $\tilde{x}=x^h$, completing the proof.

\end{proof}
\section{Proof of statement (ii)}
Let $Y^{\epsilon},Z^{\epsilon}$ be defined as in Section 2. We have the following result.
\begin{prop}\label{prop4}
Suppose the Assumptions \ref{ass0}, \ref{ass1}, \ref{ass2}, and \ref{ass3}  are satisfied, then $\rho(Y^{\epsilon},Z^{\epsilon})\rightarrow 0$ in probability.
\end{prop}
\begin{proof} Recall that $\epsilon_0$ is the constant appeared in Assumption \ref{ass1}.
For $R>0$, $0<p\leq \epsilon_0$, define
\begin{equation*}
\tau^{\epsilon}_R=\inf\{t\geq0:|Y^{\epsilon}(t)|\geq R\},\quad \tau^\epsilon_{p}=\inf\{t\geq0:|Y^{\epsilon}(t)-Z^{\epsilon}(t)|^2\geq p\}.
\end{equation*}
By (\ref{eq bound}), there exists a constant $L>0$ such  that
$$\sup_{\epsilon>0}\sup_{t\in[0,T]}|Z^{\epsilon}(t)|\leq L.$$

 We first prove that, for  $R>2L$,
\begin{equation}\label{43}
  \lim_{\epsilon\rightarrow 0}\PP\big(\tau^\epsilon_{p}\leq \tau^{\epsilon}_R\wedge T\big)=0.
\end{equation}
Set $\phi_{\delta}(y)=\int_0^y\frac{1}{\eta_R(s)+\delta}ds$. Applying the It\^{o} formula gives
\begin{eqnarray}\label{44}
 &&\hspace{-1truecm}e^{-\int_0^t|h^\epsilon(s)|^2ds}|Y^\epsilon(t)-Z^\epsilon(t)|^2\\\nonumber
 &=&\!\!\! \int_0^te^{-\int_0^s|h^\epsilon(r)|^2dr} \cdot\Big\{-|h^\epsilon(s)|^2|Y^\epsilon(s)-Z^\epsilon(s)|^2
 +
 2\langle Y^\epsilon(s)-Z^\epsilon(s),b(s,Y^\epsilon(s))-b(s,Z^\epsilon(s))\rangle\\\nonumber
  &&\ \ \ \ \ \ \ \  +2\langle Y^\epsilon(s)-Z^\epsilon(s),[\sigma(s,Y^\epsilon(s))-\sigma(s,Z^\epsilon(s))]\cdot h^\epsilon(s)\rangle
                     +\epsilon\|\sigma(s,Y^\epsilon(s))\|^2\Big\}ds\\\nonumber
  &&+2\sqrt{\epsilon}\int_0^te^{-\int_0^s|h^\epsilon(r)|^2dr}\langle\sigma(s,Y^\epsilon(s)),Y^\epsilon(s)-Z^\epsilon(s)\rangle dB(s).
\end{eqnarray}

Since $\phi_\delta(x)$ is a concave function on the interval $[0,\epsilon_0)$ and $\lim_{x\rightarrow 0}\phi^\prime_\delta(x)=\frac{1}{\delta}$, there exists a concave extension of $\phi_\delta(x)$ on the real line denoted by $\bar\phi_\delta(x)$ satisfying $\bar\phi_\delta(x)=\phi_\delta(x)$ on $[0,\epsilon_0)$. The second order derivative $\phi^{\prime\prime}_\delta(x)$ of  $\phi_\delta(x)$ in the sense of distributions is a non-positive Radon measure. Applying the It\^{o}-Tanaka formula to \eqref{44} gives
\begin{eqnarray*}
 &&\bar\phi_\delta(e^{-\int_0^t|h^\epsilon(s)|^2ds}|Y^\epsilon(t)-Z^\epsilon(t)|^2)\\
 &=& \int_0^t\bar\phi^{\prime}_\delta(e^{-\int_0^s|h^\epsilon(r)|^2dr}|Y^\epsilon(s)-Z^\epsilon(s)|^2)\cdot e^{-\int_0^s|h^\epsilon(r)|^2dr} \cdot\{-|h^\epsilon(s)|^2|Y^\epsilon(s)-Z^\epsilon(s)|^2\\
 &&+2\langle Y^\epsilon(s)-Z^\epsilon(s),b(s,Y^\epsilon(s))-b(s,Z^\epsilon(s))\rangle\\
  &&+2\langle Y^\epsilon(s)-Z^\epsilon(s),[\sigma(s,Y^\epsilon(s))-\sigma(s,Z^\epsilon(s))]\cdot h^\epsilon(s)\rangle+\epsilon\|\sigma(s,Y^\epsilon(s))\|^2\}ds\\
  &&+\sqrt{\epsilon}\int_0^t\big[\bar\phi^{\prime}_\delta(e^{-\int_0^s|h^\epsilon(r)|^2dr}|Y^\epsilon(s)-Z^\epsilon(s)|^2)\cdot e^{-\int_0^s|h^\epsilon(r)|^2dr}\cdot\langle\sigma(s,Y^\epsilon(s)),Y^\epsilon(s)-Z^\epsilon(s)\rangle\big]dB(s)\\
  &&+\int_{-\infty}^{+\infty}\Lambda_t(x)\bar\phi^{\prime\prime}_\delta(dx),
\end{eqnarray*}
where the nonnegative random field $\Lambda=\{\Lambda_t(x,\omega);(t,x,\omega)\in[0,+\infty)\times \mathbb{R}\times\Omega\}$ is the local time of the semimartingale $e^{-\int_0^\cdot|h^\epsilon(s)|^2ds}|Y^\epsilon(\cdot)-Z^\epsilon(\cdot)|^2$. Noticing
$$
\int_{-\infty}^{+\infty}\Lambda_t(x)\bar\phi^{\prime\prime}_\delta(dx)\leq 0,
$$
we derive that , for the  stopping time $\hat{\tau}^\epsilon:=T\wedge \tau^\epsilon_R\wedge\tau^\epsilon_{p}$,
\begin{eqnarray*}
&&\bar\phi_\delta(e^{-\int_0^{\hat{\tau}^\epsilon}|h^\epsilon(s)|^2ds}|Y^\epsilon(\hat{\tau}^\epsilon)-Z^\epsilon(\hat{\tau}^\epsilon)|^2)\\
 &\leq& \int_0^{\hat{\tau}^\epsilon}\bar\phi^{\prime}_\delta(e^{-\int_0^s|h^\epsilon(r)|^2dr}|Y^\epsilon(s)-Z^\epsilon(s)|^2)\cdot e^{-\int_0^s|h^\epsilon(r)|^2dr} \cdot\{-|h^\epsilon(s)|^2|Y^\epsilon(s)-Z^\epsilon(s)|^2\\
 &&+2\langle Y^\epsilon(s)-Z^\epsilon(s),b(s,Y^\epsilon(s))-b(s,Z^\epsilon(s))\rangle\\
  &&+2\langle Y^\epsilon(s)-Z^\epsilon(s),[\sigma(s,Y^\epsilon(s))-\sigma(s,Z^\epsilon(s))]\cdot h^\epsilon(s)\rangle+\epsilon\|\sigma(s,Y^\epsilon(s))\|^2\}ds\\
  &&+\sqrt{\epsilon}\int_0^{\hat{\tau}^\epsilon}\big[\bar\phi^{\prime}_\delta(e^{-\int_0^s|h^\epsilon(r)|^2dr}|Y^\epsilon(s)-Z^\epsilon(s)|^2)\cdot e^{-\int_0^s|h^\epsilon(r)|^2dr}\cdot\langle\sigma(s,Y^\epsilon(s)),Y^\epsilon(s)-Z^\epsilon(s)\rangle\big]dB(s)\\
  &\leq&\int_0^{\hat{\tau}^\epsilon}\bar\phi^{\prime}_\delta(e^{-\int_0^s|h^\epsilon(r)|^2dr}|Y^\epsilon(s)-Z^\epsilon(s)|^2)\cdot e^{-\int_0^s|h^\epsilon(r)|^2dr} \cdot\{\epsilon\|\sigma(s,Y^\epsilon(s))\|^2\\
 &&+2\langle Y^\epsilon(s)-Z^\epsilon(s),b(s,Y^\epsilon(s))-b(s,Z^\epsilon(s))\rangle+\|\sigma(s,Y^\epsilon(s))-\sigma(s,Z^\epsilon(s))\|^2\}ds\\
  &&+\sqrt{\epsilon}\int_0^{\hat{\tau}^\epsilon}\big[\bar\phi^{\prime}_\delta(e^{-\int_0^s|h^\epsilon(r)|^2dr}|Y^\epsilon(s)-Z^\epsilon(s)|^2)\cdot e^{-\int_0^s|h^\epsilon(r)|^2dr}\cdot\langle\sigma(s,Y^\epsilon(s)),Y^\epsilon(s)-Z^\epsilon(s)\rangle\big]dB(s)\\
  &\leq&\int_0^{\hat{\tau}^\epsilon}\bar\phi^{\prime}_\delta(e^{-\int_0^s|h^\epsilon(r)|^2dr}|Y^\epsilon(s)-Z^\epsilon(s)|^2)\cdot e^{-\int_0^s|h^\epsilon(r)|^2dr} \cdot[\epsilon\|\sigma(s,Y^\epsilon(s))\|^2\\
  &&+g(s)\eta_R(|Y^\epsilon(s)-Z^\epsilon(s)|^2)]ds\\
  &&+\sqrt{\epsilon}\int_0^{\hat{\tau}^\epsilon}\big[\bar\phi^{\prime}_\delta(e^{-\int_0^s|h^\epsilon(r)|^2dr}|Y^\epsilon(s)-Z^\epsilon(s)|^2)\cdot e^{-\int_0^s|h^\epsilon(r)|^2dr}\cdot\langle\sigma(s,Y^\epsilon(s)),Y^\epsilon(s)-Z^\epsilon(s)\rangle\big]dB(s).
\end{eqnarray*}
Taking expectation we get 
\begin{eqnarray}\label{eq Tn}
&&\mathbb{E}\phi_\delta(e^{-\int_0^{\hat{\tau}^\epsilon}|h^\epsilon(s)|^2ds}|Y^\epsilon(\hat{\tau}^\epsilon)-Z^\epsilon(\hat{\tau}^\epsilon)|^2)\nonumber\\
&\leq&\mathbb{E}\int_0^{\hat{\tau}^\epsilon}\phi^{\prime}_\delta(e^{-\int_0^s|h^\epsilon(r)|^2dr}|Y^\epsilon(s)-Z^\epsilon(s)|^2)\cdot e^{-\int_0^s|h^\epsilon(r)|^2dr} \cdot[\epsilon\|\sigma(s,Y^\epsilon(s))\|^2\nonumber\\
  &&+g(s)\eta_R(|Y^\epsilon(s)-Z^\epsilon(s)|^2)]ds\nonumber\\
&\leq& C\int_0^Tg(s)ds+\frac{\epsilon}{\delta}\mathbb{E}\int_0^{\hat{\tau}^\epsilon}\|\sigma(s,Y^\epsilon(s))\|^2ds.
\end{eqnarray}
The last inequality follows from the property of $\eta_R$ stated in Assumption \ref{ass3}. Hence by taking $\epsilon\rightarrow 0$, we obtain
\begin{equation*}
  \limsup_{\epsilon\rightarrow 0}\mathbb{E}\phi_\delta(e^{-\int_0^{\hat{\tau}^\epsilon}|h^\epsilon(s)|^2ds}|Y^\epsilon(\hat{\tau}^\epsilon)-Z^\epsilon(\hat{\tau}^\epsilon)|^2)\leq C\int_0^Tg(s)ds,
\end{equation*}
that is,
\begin{equation}\label{45}
 \limsup_{\epsilon\rightarrow 0}\mathbb{E}\int_0^{e^{-\int_0^{\hat{\tau}^\epsilon}|h^\epsilon(s)|^2ds}|Y^\epsilon(\hat{\tau}^\epsilon)-Z^\epsilon(\hat{\tau}^\epsilon)|^2}\frac{1}{\eta_R(u)+\delta}du\leq C\int_0^Tg(s)ds.
\end{equation}
Since for  $\eta\leq \epsilon_0$,
\begin{equation}\label{46}
\int_0^{e^{-N}\eta}\frac{1}{\eta_R(s)+\delta}ds\PP\big(|Y^\epsilon(\hat{\tau}^\epsilon)-Z^\epsilon(\hat{\tau}^\epsilon)|^2\geq \eta\big)
\leq\mathbb{E}\int_0^{e^{-\int_0^{\hat{\tau}^\epsilon}|h^\epsilon(s)|^2ds}|Y^\epsilon(\hat{\tau}^\epsilon)-Z^\epsilon(\hat{\tau}^\epsilon)|^2}
\frac{1}{\eta_R(s)+\delta}ds, \end{equation}
we assert, by (\ref{24}) and (\ref{45}), that
\begin{equation}\label{47}
 \limsup_{\epsilon\rightarrow 0}\PP\big(|Y^\epsilon(\hat{\tau}^\epsilon)-Z^\epsilon(\hat{\tau}^\epsilon)|^2\geq \eta\big)
 \leq
 \lim_{\delta\downarrow 0}  \frac{C\int_0^Tg(s)ds}{\int_0^{e^{-N}\eta }\frac{1}{\eta_R(s)+\delta}ds}
 =
 0.
\end{equation}

Noting that
$$\{\sup_{s\leq T\wedge\tau^\epsilon_R\wedge\tau^\epsilon_{p}}|Y^\epsilon(s)-Z^\epsilon(s)|^2\geq p\}\subset \{|Y^\epsilon(\hat{\tau}^\epsilon)-Z^\epsilon(\hat{\tau}^\epsilon)|^2\geq p\},$$
it follows from (\ref{47}) that
\begin{equation}\label{48}
\lim_{\epsilon \rightarrow 0}\PP\big(\sup_{s\leq T\wedge\tau^\epsilon_R\wedge\tau^\epsilon_{p}}|Y^\epsilon(s)-Z^\epsilon(s)|^2\geq p\big)=0.
\end{equation}
Now  \eqref{43} follows from the inclusion:
\begin{equation*}
\{\tau^\epsilon_{p}\leq T\wedge\tau^{\epsilon}_R\}\subset \{\sup_{s\leq T\wedge\tau^\epsilon_R\wedge\tau^\epsilon_{p}}|Y^\epsilon(s)-Z^\epsilon(s)|^2\geq p\}.
\end{equation*}
\vskip 0.1cm

Next, we prove
\begin{equation}\label{50}
\lim_{R\rightarrow\infty}\sup_{\epsilon\in(0,1)}\PP\big(\tau^\epsilon_R \leq T\wedge\tau^\epsilon_{p}\big)=0.
\end{equation}

Let $\varphi(x)$ be defined as that in Proposition \ref{prop3} and denote $W^{h^\epsilon}(t)=e^{-\eta\int_0^{t}|h^\epsilon(s)|^2ds}V(Y^\epsilon(t))$. Similar to (\ref{eq Tn}), applying the It\^{o}-Tanaka formula gives
\begin{eqnarray}
 &&\mathbb{E}\varphi(W^{h^\epsilon}(T\wedge\tau^\epsilon_R\wedge\tau^\epsilon_{p}))\nonumber\\
 &\leq& \varphi(V(x_0))+\mathbb{E}\int_0^{T\wedge\tau^\epsilon_R\wedge\tau^\epsilon_{p}} \varphi^\prime(W{^{h^\epsilon}}(s))\cdot e^{-\eta\int_0^s|h^\epsilon(r)|^2dr} \cdot\big[-\eta|h^\epsilon(s)|^2V(Y^\epsilon(s))\nonumber\\
 &&+\langle V_x(Y^\epsilon(s)),b(s,Y^\epsilon(s))\rangle+\langle V_x(Y^\epsilon(s)),\sigma(s,Y^\epsilon(s))\cdot h^\epsilon(s)\rangle\nonumber\\
 &&+\epsilon\cdot\text{trace}\big(V_{xx}(Y^\epsilon(s))\sigma(s,Y^\epsilon(s))\sigma^T(s,Y^\epsilon(s))\big)\big]ds\nonumber\\
  &\leq& \varphi(V(x_0))+\mathbb{E}\int_0^{T\wedge\tau^\epsilon_R\wedge\tau^\epsilon_{p}} \varphi^\prime(W{^{h^\epsilon}}(s))\cdot e^{-\eta\int_0^s|h^\epsilon(r)|^2dr} \cdot\big[-\eta|h^\epsilon(s)|^2V(Y^\epsilon(s))\nonumber\\
 &&+\langle V_x(Y^\epsilon(s)),b(s,Y^\epsilon(s))\rangle\nonumber+\frac{|\langle\sigma(s,Y^\epsilon(s)),V_x(Y^\epsilon(s))\rangle|}{\sqrt{\eta V(Y^\epsilon(s))}}\cdot \sqrt{\eta V(Y^\epsilon(s))}|h^\epsilon(s)|\big]ds\nonumber\\
 &&+\epsilon\cdot\text{trace}\big(V_{xx}(Y^\epsilon(s))\sigma(s,Y^\epsilon(s))\sigma^T(s,Y^\epsilon(s))\big)\nonumber\\
 &\leq&\varphi(V(x_0))+\mathbb{E}\int_0^{T\wedge\tau^\epsilon_R\wedge\tau^\epsilon_{p}} \varphi^\prime(W{^{h^\epsilon}}(s))\cdot e^{-\eta\int_0^s|h^\epsilon(r)|^2dr} \cdot\big[\langle V_x(Y^\epsilon(s)),b(s,Y^\epsilon(s))\rangle\nonumber\\
&&+\frac{|\langle\sigma(s,Y^\epsilon(s)),V_x(Y^\epsilon(s))\rangle|^2}{\eta V(Y^\epsilon(s))}+\epsilon\cdot\text{trace}\big(V_{xx}(Y^\epsilon(s))\sigma(s,Y^\epsilon(s))\sigma^T(s,Y^\epsilon(s))\big)\big]ds\nonumber\\
 &\leq& \varphi(V(x_0))+\mathbb{E}\int_0^{T\wedge\tau^\epsilon_R\wedge\tau^\epsilon_{p}} \varphi^\prime(W^{h^\epsilon}(s))\cdot e^{-\eta\int_0^s|h^\epsilon(r)|^2dr} \cdot f(s)\big(1+\gamma(V(Y^{\epsilon}(s)))\big)ds\nonumber\\
 &\leq& \varphi(V(x_0))+C\int_0^Tf(s)ds.\label{3-4}
\end{eqnarray}
The last inequality follows from the property of $\gamma$ in Assumption \ref{ass3}. By \eqref{3-4} and the definition of $\varphi$, we deduce that
\begin{equation*}
\PP\big(\tau^\epsilon_R\leq T\wedge\tau^\epsilon_p\big)\leq\frac{\varphi(V(x_0))+C\int_0^Tf(s)ds}{\int_0^{e^{-\eta N}\cdot V(R)}\frac{1}{\gamma(s)+1}ds}.
\end{equation*}
Finally, by (\ref{As b V}) and (\ref{26}), letting $R\rightarrow\infty$, we obtain \eqref{50}.

Now  we are in the position to complete the proof of the  theorem.

For arbitrary $\delta_1>0$, we have
\begin{eqnarray*}
&&\PP\big(\sup_{0\leq s\leq T}|Y^\epsilon(s)-Z^\epsilon(s)|\geq\delta_1\big)\\
&=&\PP\big(\sup_{0\leq s\leq T}|Y^\epsilon(s)-Z^\epsilon(s)|\geq\delta_1,\tau^\epsilon_R\wedge\tau^\epsilon_{\delta_1^2}>T\big)\\
&&+\PP\big(\sup_{0\leq s\leq T}|Y^\epsilon(s)-Z^\epsilon(s)|\geq\delta_1,\tau^\epsilon_R\wedge\tau^\epsilon_{\delta_1^2}\leq T\big)\\
&=&\PP\big(\sup_{0\leq s\leq T}|Y^\epsilon(s)-Z^\epsilon(s)|\geq\delta_1,\tau^\epsilon_R\wedge\tau^\epsilon_{\delta_1^2}>T\big)\\
&&+\PP\big(\sup_{0\leq s\leq T}|Y^\epsilon(s)-Z^\epsilon(s)|\geq\delta_1,\tau^\epsilon_R\leq T\wedge\tau^\epsilon_{\delta_1^2}\big)\\
&&+\PP\big(\sup_{0\leq s\leq T}|Y^\epsilon(s)-Z^\epsilon(s)|\geq\delta_1,\tau^\epsilon_{\delta_1^2}\leq \tau^\epsilon_R\wedge T\big)\\
&\leq&\PP\big(\sup_{0\leq s\leq T\wedge\tau^\epsilon_R\wedge\tau^\epsilon_{\delta_1^2}}|Y^\epsilon(s)-Z^\epsilon(s)|^2\geq \delta_1^2\big)\\
&&+\PP\big(\tau^\epsilon_R\leq T\wedge\tau^\epsilon_{\delta_1^2}\big)+\PP\big(\tau^\epsilon_{\delta_1^2}\leq \tau^\epsilon_R\wedge T\big).
\end{eqnarray*}
\eqref{43} and \eqref{48}(with $p=\delta_1^2$) imply that, for any $R>2L$,
\begin{eqnarray*}
\lim_{\epsilon\rightarrow 0}\PP\big(\sup_{0\leq s\leq T}|Y^\epsilon(s)-Z^\epsilon(s)|\geq\delta_1\big)
\leq
\sup_{\epsilon\in(0,1)}\PP\big(\tau^\epsilon_R\leq T\wedge\tau^\epsilon_{\delta_1^2}\big).
\end{eqnarray*}
Let $R\rightarrow \infty$ to get
\begin{equation*}
\lim_{\epsilon\rightarrow0}\PP\big(\sup_{0\leq s\leq T}|Y^\epsilon(s)-Z^\epsilon(s)|\geq\delta_1\big)=0.
\end{equation*}

The proof is complete.
\end{proof}
\section{Applications and examples}
In this part, we present some applications. The two Assumptions \ref{ass0} and \ref{ass3} are all satisfied by the following examples. Therefore, the large deviation principle holds for the models.

\noindent\textbf{Example 5.1.}
The following one-dimensional SDE is also considered in \cite{wu}:
\begin{equation}\label{51}
dx(t)=-x^{\frac{1}{3}}(t)dt+x^{\frac{2}{3}}(t)dB(t).
\end{equation}
Note that both the drift term and the diffusion term are H\"{o}lder continuous.
Indeed,
\begin{eqnarray}
&&2\langle x-y, b(x) -b(y)\rangle + \| \sigma(x)- \sigma(y)\|^2 \notag \\
&=& (x^{\frac{2}{3}}-y^{\frac{2}{3}})^2 - 2(x-y)(x^{\frac{1}{3}}-y^{\frac{1}{3}}) \notag \\
&=& (x^{\frac{1}{3}}-y^{\frac{1}{3}})^2 \Big[ (x^{\frac{1}{3}}+y^{\frac{1}{3}})^2 - 2(x^{\frac{1}{3}}+x^{\frac{1}{3}}y^{\frac{1}{3}}+ y^{\frac{1}{3}})\Big]\notag \\
&=& - (x^{\frac{1}{3}}-y^{\frac{1}{3}})^2 (x^{\frac{2}{3}}+y^{\frac{2}{3}}) \notag \\
&\leq& \et_R(|x-y|^2).
\end{eqnarray}
If we take $\delta=1$, $\eta=4$ and $V(x)=x^2$ in Assumption \ref{ass2}, then
 \begin{equation*}
   \langle b(s,x),V_x(x)\rangle+\frac{\delta}{2}trace\big(V_{xx}(x)\sigma(s,x)\sigma^T(s,x)\big)+\frac{|\langle\sigma(s,x),V_x (x)\rangle|^2}{\eta V(x)} = 0\leq \ga(|x|^2).
 \end{equation*}
So Assumption \ref{ass1}, \ref{ass2} holds. Hence, Theorem \ref{thm} holds.

\noindent\textbf{Example 5.2.}
The following multi-dimensional SDEs are also considered in \cite[example 171]{R-SITU}:
\begin{equation}\label{53}
dx(t)=-x(t)|x(t)|^{-\alpha}dt+\sigma(x(t))dB(t),
\end{equation}
where $\alpha\in(0,1),~\sigma$ is local Lipschitz continuous. It is easy to check that the function $x|x|^{-\al}$ is not locally Lipschitz. However it satisfies
\begin{eqnarray*}
  \langle x-y,-x|x|^{-\al}+y|y|^{-\al}\rangle&=& -|x|^{2-\al}+\langle y,x|x|^{-\al}\rangle+ \langle x,y|y|^{-\al}\rangle-|y|^{2-\al}\\
  &\leq&-|x|^{2-\al}-|y|^{2-\al}+|y||x|^{1-\al}+|x||y|^{1-\al}\\
  &=&(|x|-|y|)(|y|^{1-\al}-|x|^{1-\al})\leq 0.
\end{eqnarray*}
Thus, if we consider an equation of the form (\ref{53}) with any locally Lipschitz coefficient $\sigma$ which satisfies Assumption \ref{ass2}, then all the conditions in Theorem \ref{thm} are ensuerd.

\noindent\textbf{Example 5.3.} (Stochastic Duffing-van der Pol oscillator model)
The Duffing-van der Pol equation unifies both the Duffing equation and the van der Pol equation and has been used for example in certain flow-induced structural vibration problems \cite{holmes}. Cox \& Hutzenthaler \& Jentzen has considered the following more general stochastic model in \cite{jentzen}:
\begin{eqnarray*}
\ddot{X}_t^{x,1}&=& \alpha_2\dot{X}_t^{x,1}-\alpha_1X_t^{x,1}-\alpha_3(X_t^{x,1})^2\dot{X}_t^{x,1}-(X_t^{x,1})^3+g(X_t^{x,1})\dot{W}_t,  \\
X^{x,1}_0&=&x_1,~~\dot{X}^{x,1}_0~=~x_2,
\end{eqnarray*}
where $\alpha_1,~\alpha_2,~\alpha_3\in(0,\infty)$. Here we assume $|g(x)|^2\leq \eta_0+\eta_1|x|^4$, $\eta_0,~\eta_1>0$.

By defining $\dot{X}_t^{x,1}=X_t^{x,2}$, then the above equation can be transformed equivalently into the following SDEs:
\begin{eqnarray*}
dX^{x,1}_t&=&X_t^{x,2}dt,\\
dX_t^{x,2}&=& \big[\alpha_2X_t^{x,2}-\alpha_1X_t^{x,1}-\alpha_3(X_t^{x,1})^2X_t^{x,2}-(X_t^{x,1})^3\big]dt+g(X_t^{x,1})dW_t,\\
X^{x,1}_0&=&x_1,~~X^{x,2}_0~=~x_2.
\end{eqnarray*}

 For $x=(x_1,x_2)\in \mathbb{R}^2$, set $b(s,x)=(x_2,\alpha_2x_2-\alpha_1x_1-\alpha_3(x_1)^2x_2-(x_1)^3)^T$ and $\sigma(s,x)=(0,g(x_1))^T$. Since $b$ and $\sigma$ are local Lipschitz continuous, Assumption \ref{ass1} holds if $\eta_R(s)=L_R s$, where $L_R$ is a constant only dependent on $R$. Define $V(x)=\frac{(x_1)^4}{2}+\alpha_1(x_1)^2+(x_2)^2$ and set $\delta=\eta=1$ in Assumption \ref{ass2}, then
 \begin{eqnarray*}
   &&\langle b(s,x),V_x(x)\rangle+\frac{1}{2}trace\big(V_{xx}(x)\sigma(s,x)\sigma^T(s,x)\big)+\frac{|\langle\sigma(s,x),V_x (x)\rangle|^2}{V(x)}\\
   &=&x_2(2(x_1)^3+2\alpha_1x_1)+2x_2(\alpha_2x_2-\alpha_1x_1-\alpha_3(x_1)^2x_2-(x_1)^3)+|g(x_1)|^2\\
   &&+\frac{4(x_2)^2|g(x_1)|^2}{\frac{(x_1)^4}{2}+\alpha_1(x_1)^2+(x_2)^2}\\
   &=&2\alpha_2(x_2)^2-2\alpha_3(x_1)^2(x_2)^2+|g(x_1)|^2+\frac{4(x_2)^2|g(x_1)|^2}{\frac{(x_1)^4}{2}+\alpha_1(x_1)^2+(x_2)^2}\\
   &\leq&\eta_0+2\alpha_2(x_2)^2+\eta_1(x_1)^4+\frac{4\eta_0(x_2)^2+4\eta_1(x_1)^4(x_2)^2}{\frac{(x_1)^4}{2}+\alpha_1(x_1)^2+(x_2)^2}\\
   &\leq&5\eta_0+(8\eta_1+2\alpha_2)(x_2)^2+\eta_1(x_1)^4\\
   &\leq&K(1+V(x)),
 \end{eqnarray*}
 where $K=5\eta_0+10\eta_1+2\alpha_2$. Hence Assumption \ref{ass2} also holds by taking $\gamma(s)=s$.
Then, Theorem \ref{thm} holds.

\noindent\textbf{Example 5.4.} (Stochastic SIR model)
The SIR model from epidemiology for the total number of susceptible, infected and recovered individuals has been introduced by Anderson \& May \cite{anderson}. Here we consider the following stochastic SIR model:
\begin{eqnarray*}
dX^{x,1}_t&=&(-\alpha X^{x,1}_tX^{x,2}_t-\kappa X^{x,1}_t+\kappa)dt-\beta X^{x,1}_tX^{x,2}_tdW^1_t,\\
dX_t^{x,2}&=& (\alpha X^{x,1}_tX^{x,2}_t-(\gamma+\kappa)X^{x,2}_t)dt+\beta X^{x,1}_tX^{x,2}_tdW^2_t,\\
dX_t^{x,3}&=& (\gamma X^{x,2}_t-\kappa X^{x,3}_t)dt,\\
X^{x,1}_0&=&x_1,~~X^{x,2}_0~=~x_2,~~X^{x,3}_0~=~x_3,
\end{eqnarray*}
where $\alpha,~\beta,~\gamma,~\kappa\in(0,\infty)$ and $x=(x_1,x_2,x_3)\in[0,\infty)^3$.

For $x=(x_1,x_2,x_3) \in [0,\infty)^3$, set $b(s,x)=(-\alpha x_1x_2-\kappa x_1+\kappa,\alpha x_1x_2-(\gamma+\kappa)x_2,\gamma x_2-\kappa x_3)^T$ and $\sigma(s,x)=(-\beta x_1x_2,\beta x_1x_2,0)^T$. Since $b$ and $\sigma$ are local Lipschitz continuous, Assumption \ref{ass1} holds if $\eta_R(s)=L_R s$.
So to verify the Theorem \ref{thm}, we need only to verify Assumption \ref{ass2}.

Define $V(x)=(x_1+x_2-1)^2$ and let $\delta$ and $\eta$ be any positive constants in Assumption \ref{ass2}. A direct calculation gives that
\begin{equation*}
\langle b(s,x),V_x(x)\rangle+\frac{\delta}{2}trace\big(V_{xx}(x)\sigma(s,x)\sigma^T(s,x)\big)+\frac{|\langle\sigma(s,x),V_x (x)\rangle|^2}{\eta V(x)}\leq \frac{\gamma}{2}.
\end{equation*}
So Assumption \ref{ass2} holds.

\noindent\textbf{Example 5.5.} (Stochastic Lotka-Volterra (LV) systems) The well-known Lotka-Volterra systems play an important role in population dynamics, game theory, and so on (see \cite{H-S}). Here, we consider the three-dimensional Stratonovich stochastic competitive LV systems:
\begin{eqnarray*}
dy_1 &=& y_1(r-a_{11}y_1 - a_{12}y_2-a_{13}y_3)dt + \sigma \circ y_1 dB(t), \\
dy_2 &=& y_2(r-a_{21}y_1 - a_{22}y_2-a_{23}y_3)dt + \sigma \circ y_2 dB(t), \\
dy_3 &=& y_3(r-a_{31}y_1 - a_{32}y_2-a_{33}y_3)dt + \sigma \circ y_3 dB(t),
\end{eqnarray*}
where $r>0$, $a_{ij} >0$, $i, j =1, 2, 3$ and initial data $(y_1(0), y_2(0), y_3(0))\in(0,+\infty)^3$. According to \cite[Theorem 3.2]{CDJNZ}, we obtain $y(t) = (y_1(t), y_2(t), y_3(t)) \in (0, +\infty)^3$ for all $t>0$.

The coefficients of above equation are local Lipschitz continuous, so Assumption \ref{ass1} holds.
If $\min\{a_{11},a_{22},a_{33}\} > \frac{\sigma^2}{2}$, then Assumption \ref{ass2} holds with $\ga(s) = s$ and $V(y) = |y|^2$.
So, all the conditions in Theorem \ref{thm} are satisfied.
\vskip 0.5cm
\noindent{\bf Acknowledgement}. This work is partially  supported by NSFC (No. 12131019, 11971456, 11721101).

\end{document}